\documentclass[12pt]{amsart}
\pdfoutput=1

\usepackage[utf8]{inputenc}
\usepackage{indentfirst}
\usepackage{amsmath}
\usepackage{amsthm}
\usepackage{amssymb}
\usepackage{amsfonts}
\usepackage{microtype}
\usepackage{fancyhdr}
\usepackage{tikz-cd}
\usetikzlibrary{arrows}
\usetikzlibrary{positioning}
\usepackage{xcolor}
\usepackage[colorlinks]{hyperref}
\usepackage{mathtools}
\usepackage{thmtools}

\usepackage[
  margin=1in,
  includefoot,
  footskip=30pt,
]{geometry}

\usepackage{csquotes}

\usepackage[style=alphabetic, citestyle=alphabetic]{biblatex}
\hypersetup{
  colorlinks = true,
  citecolor = {purple!70!black},
  linkcolor = {red!60!black}
}

\renewcommand{\O}{\mathcal{O}}
\renewcommand{\P}{\mathbb{P}}

\newcommand{\Z}{\mathbb{Z}}

\newcommand{\Ext}{\mathrm{Ext}}
\newcommand{\Hom}{\mathrm{Hom}}
\newcommand{\RHom}{\mathrm{RHom}}

\newcommand{\RGamma}{\mathrm{R}\Gamma\,}

\newcommand{\Dbcoh}{D^b_{\!\mathrm{coh}}}

\newcommand{\mA}{\mathcal{A}}

\newcommand{\dual}{{\scriptstyle\vee}}
\newcommand{\iso}{\simeq}
\newcommand{\caniso}{\cong}

\newcommand{\sldots}{\, \ldots \,}

\newcommand{\rdim}{\operatorname{rdim}}

\newcommand{\highlight}[1]{\colorbox{gray!10}{$#1$}}

\declaretheoremstyle[
headformat=\NUMBER.\,\NAME\NOTE,
postheadspace=.5em,
spaceabove=6pt,
headfont=\normalfont\small\scshape,
notefont=\normalfont\small\mdseries, notebraces={(}{)},
bodyfont=\normalfont\itshape
]{plainswap}
\declaretheoremstyle[
headformat=\NUMBER.\,\NAME\NOTE,
postheadspace=.5em,
spaceabove=6pt,
headfont=\normalfont\small\scshape,
notefont=\normalfont\mdseries, notebraces={(}{)},
bodyfont=\normalfont
]{definitionswap}
\declaretheoremstyle[
headformat=\NAME\NOTE,
postheadspace=.5em,
spaceabove=6pt,
headfont=\normalfont\itshape,
notefont=\mdseries, notebraces={(}{)},
bodyfont=\normalfont
]{myremark}

\declaretheorem[style=plainswap, name=Theorem, sharenumber=subsection]{theorem}
\declaretheorem[style=plainswap, numberlike=theorem, name=Proposition]{proposition}
\declaretheorem[style=plainswap, numberlike=theorem, name=Lemma]{lemma}
\declaretheorem[style=plainswap, numberlike=theorem, name=Corollary]{corollary}
\declaretheorem[style=plainswap, numberlike=theorem, name=Conjecture]{conjecture}

\newtheorem*{claim}{Claim}

\theoremstyle{definition}
\declaretheorem[style=definitionswap, numberlike=theorem, name=Definition]{definition}

\theoremstyle{myremark}
\newtheorem*{remark}{Remark}

\theoremstyle{remark}

\begin{document}

\title{Rouquier dimension of some blow-ups}
\author{Dmitrii Pirozhkov}

\begin{abstract}
  Rapha\"{e}l Rouquier introduced an invariant of triangulated categories which is known as Rouquier dimension. Orlov conjectured that for any smooth quasi-projective variety $X$ the Rouquier dimension of $D^b_{\mathrm{coh}}(X)$ is equal to $\mathrm{dim}\, X$. In this note we show that some blow-ups of projective spaces satisfy Orlov's conjecture. This includes a blow-up of $\mathbb{P}^2$ in nine arbitrary distinct points, or a blow-up of three distinct points lying on an exceptional divisor of a blow-up of $\mathbb{P}^3$ in a line. In particular, our method gives an alternative proof of Orlov's conjecture for del Pezzo surfaces, first established by Ballard and Favero.
\end{abstract}

\maketitle

\section{Introduction}
Rapha\"{e}l Rouquier introduced in \cite{rouquier} an invariant of triangulated categories which became known as Rouquier dimension. For a scheme $X$ the bounded derived category of coherent sheaves $\Dbcoh(X)$ is a triangulated category which reflects many geometric properties of $X$. A natural question to ask is what is the geometric meaning of the Rouquier dimension of the category $\Dbcoh(X)$ in terms of $X$.

In general, Rouquier dimension is very difficult to determine. Rouquier proved that for any smooth variety there is an inequality $\dim X \leq \rdim \Dbcoh(X) \leq 2 \dim X$. However, in all cases where Rouquier dimenson of a smooth variety was computed exactly, it happened to coincide with the usual geometric dimension of $X$. Dmitri Orlov made a conjecture \cite{orlov-curves}:

\begin{conjecture}[Orlov]
  Let $X$ be a smooth quasi-projective variety. Then the Rouquier dimension of $\Dbcoh(X)$ equals $\dim(X)$.
\end{conjecture}

The conjecture is known in multiple cases:
\begin{itemize}
\item projective spaces, quadrics, Grassmannians \cite{rouquier};
\item del Pezzo surfaces, Hirzebruch surfaces, and toric surfaces with nef anti-canonical divisors; Fano threefolds of type $V_5$ and $V_{22}$ \cite{ballard-favero};
\item direct products of the varieties above \cite{yang};
\item all smooth proper curves \cite{orlov-curves};
\item a product of two Fermat elliptic curves and a Fermat K3 surface \cite[Th.~1.6]{ballard-favero-katzarkov}.
\end{itemize}

In this note we give new examples of varieties satisfying Orlov's conjecture. Our examples are some particular blow-ups of projective spaces. This gives an alternative, easier proof for del Pezzo surfaces, but also covers some higher-dimensional cases. In an arbitrary dimension we show the following:

{
  \hypersetup{hidelinks}
  \renewcommand{\thetheorem}{\ref{thm: blowing up projective spaces}}
  \begin{theorem}
    Let $\{ Z_b \}_{b \in B}$ be a set of at most three disjoint linear subspaces of $\P^n$  such that every subspace $Z_b$ is either a point or has codimension two. Denote by $Y$ the blow-up of the projective space in the union $\sqcup_{b \in B} Z_b$. Then the variety $Y$ satisfies Orlov's conjecture.
  \end{theorem}
  \addtocounter{theorem}{-1}
}

Let us illustrate the idea of the proof. If a category admits a full exceptional collection of length $n+1$, then each exceptional object generates a subcategory of Rouquier dimension zero, and the standard estimate on the Rouquier dimension of the glueing (Lemma~\ref{lem: rouquier dimension under glueing}) bounds it from above by $n$. This is sufficient to establish Orlov's conjecture for projective spaces, but not enough for any blow-up, since the exceptional collection becomes too long. However, a category generated by a single exceptional object is not the only triangulated category of Rouquier dimension zero. The derived category of representations of any ADE quiver is another example (Lemma~\ref{lem: dynkin quivers have rouquier dimension zero}). In the proof of Theorem~\ref{thm: blowing up projective spaces} we construct a semiorthogonal decomposition for the derived category of the blow-up with $2$ exceptional objects and $n-1$ components equivalent to the derived category of the $D_4$ quiver, and conclude by the same estimate for the glueing.

The procedure outlined above preserves two exceptional objects from the original full exceptional collection on the projective space. In the cases of low dimension, i.e., $\P^2$ and $\P^3$, this may be used to bound the Rouquier dimension for a tower of blow-ups, with three levels for $\P^2$ and two levels for $\P^3$:

{
  \hypersetup{hidelinks}
  \renewcommand{\theproposition}{\ref{prop: surfaces satisfying orlov conjecture}}
  \begin{proposition}
    Consider a tower of maps
    \[
      X_3 \to X_2 \to X_1 \to X_0 = \P^2
    \]
    where each map $\pi_i\colon X_i \to X_{i-1}$ is a blow-up in at most three distinct points. Then $X_3$ satisfies Orlov's conjecture, i.e., $\rdim(X_3) = 2$.
  \end{proposition}
  \addtocounter{proposition}{-1}
}
{
  \hypersetup{hidelinks}
  \renewcommand{\theproposition}{\ref{prop: threefolds satisfying orlov conjecture}}
  \begin{proposition}
    Consider a tower of maps
    \[
      X_2 \to X_1 \to X_0 = \P^3
    \]
    where each map $\pi_i\colon X_i \to X_{i-1}$ is a blow-up of a disjoint union of points and lines, at most three per level, where by a line we mean a strict transform of an one-dimensional linear subspace in $\P^3$. Then $X_2$ satisfies Orlov's conjecture, i.e., $\rdim(X_2) = 3$.
  \end{proposition}
  \addtocounter{proposition}{-1}
}

\vspace{2mm}

The paper is organized as follows. In Section~\ref{sec: preliminaries} we introduce standard definitions related to semiorthogonal decompositions and Rouquier dimension. Section~\ref{sec: blowing up lemmas} describes some subcategories of Rouquier dimension zero which naturally arise inside the derived categories of certain blow-ups. In Section~\ref{sec: blow-ups of projective spaces} we prove the main results of this note, stated above. Finally, Appendix~\ref{sec: dual to orlov collection} contains a folklore result describing the full mutation of Orlov's semiorthogonal decomposition for a blow-up, which is needed in Section~\ref{sec: blowing up lemmas}.

\subsection*{Acknowledgements} The author would like to thank Alexander Kuznetsov for guidance and many helpful discussions.

\section{Preliminaries}
\label{sec: preliminaries}
All varieties and triangulated categories in this note are over a field $k$. All functors are assumed to be derived functors, with the exception of the global sections functor $\Gamma$. The (hyper)cohomology is always denoted by $\RGamma$.

Let $\mathcal{T}$ be a triangulated category. In this note we only consider $\Ext$-finite triangulated categories. For a pair of objects $A, B \in \mathcal{T}$ we denote by $\RHom_{\mathcal{T}}(A, B)$ the graded vector space $\oplus_i \Hom_{\mathcal{T}}(A, B[i])$.

A triangulated subcategory $\mA \subset \mathcal{T}$ is \emph{admissible} if its inclusion functor has both adjoints. A pair $\langle \mA, \mathcal{B} \rangle$ of triangulated subcategories of $\mathcal{T}$ forms a \emph{semiorthogonal decomposition} if both of them are admissible, they generate $\mathcal{T}$, and for any pair of objects $A \in \mA$, $B \in \mathcal{B}$ we have $\RHom_T(B, A) = 0$. In more general situations the admissibility condition is usually weakened, but for the categories occuring in our paper this definition is equivalent to the standard one.

For any semiorthogonal decomposition $T = \langle \mA, \mathcal{B} \rangle$ there exists another semiorthogonal decomposition $T = \langle \widetilde{\mathcal{B}}, \mA \rangle$ such that the category $\widetilde{\mathcal{B}}$ is equivalent to~$\mathcal{B}$ \cite{bondal-exceptional}. We refer to the subcategory $\widetilde{\mathcal{B}} \subset \mathcal{T}$ as the \emph{mutation of $\mathcal{B}$ through $\mathcal{A}$}. There is an analogous notion of a mutation of $\mA$ through $\mathcal{B}$.

A semiorthogonal decomposition into several components $\mathcal{T} = \langle \mA_1, \mA_2, \ldots, \mA_k \rangle$ is defined similarly. Given such a decomposition, its (right) \emph{dual semiorthogonal decomposition} is the one obtained by the following sequence of mutations. Mutate $\mA_k$ through $\langle \mA_1, \ldots, \mA_{k-1} \rangle$, then mutate $\mA_{k-1}$ through $\langle \mA_1, \ldots, \mA_{k-2} \rangle$, and so on:
\[
  \begin{aligned}
    \langle \mA_1, \sldots, \mA_{k-1}, \, \mA_k \rangle & \rightsquigarrow
    \langle \widetilde{\mA_k}, \mA_1, \sldots, \mA_{k-1} \rangle \rightsquigarrow
    \langle \widetilde{\mA_{k}}, \widetilde{\mA_{k-1}}, \mA_1, \sldots, \mA_{k-2} \rangle \rightsquigarrow \\
    & \rightsquigarrow \sldots \rightsquigarrow \langle \widetilde{\mA_k}, \widetilde{\mA_{k-1}}, \sldots, \widetilde{\mA_1} \rangle.
  \end{aligned}
\]
As in the case with a semiorthogonal decomposition into two components, for every $i$ the category $\widetilde{\mA_i}$ is equivalent to $\mA_i$.

To introduce Rouquier dimension, we need some notation.
For any object $E \in \mathcal{T}$ we denote by $\langle E \rangle_0 \subset \mathcal{T}$ the subcategory whose objects are direct summands of finite direct sums of shifts of $E$. The subcategories $\langle E \rangle_n$ are defined inductively. An object $F \in \mathcal{T}$ lies in the subcategory $\langle E \rangle_n$ if and only if there exists a distinguished triangle in $\mathcal{T}$:
\[
  A \to B \to C \to A[1]
\]
where $A \in \langle E \rangle_0$, $B \in \langle E \rangle_{n-1}$, and $C$ has a direct summand isomorphic to $F$. Note that the union $\cup_{n \geq 0} \langle E \rangle_n$ is the smallest triangulated subcategory of $\mathcal{T}$ which contains $E$ and is Karoubian closed.

\begin{definition}[\cite{rouquier}]
  Let $\mathcal{T}$ be a triangulated category. Its \emph{Rouquier dimension} $\rdim \mathcal{T}$ is the smallest number $n$ such that there exists an object $E \in \mathcal{T}$ with $\langle E \rangle_n = \mathcal{T}$. If there are no such objects for any $n$, we set $\rdim \mathcal{T} = \infty$.
\end{definition}

If $X$ is a smooth variety, we denote by $\rdim X$ the Rouquier dimension of $\Dbcoh(X)$.

\begin{lemma}[\cite{rouquier}]
  \label{lem: rouquier dimension under glueing}
  Let $\mathcal{T}$ be a triangulated category with a semiorthogonal decomposition
  \[
    \mathcal{T} = \langle \mA_1, \ldots, \mA_n \rangle.
  \]
  Then $\rdim \mathcal{T} \leq \sum_{i=1}^n \rdim(\mA_i) + n - 1$.
\end{lemma}
\begin{proof}
  It is enough to prove this for a decomposition into two components $\mathcal{T} = \langle \mA_1, \mA_2 \rangle$, which is \cite[Lem.~3.5]{rouquier}.
\end{proof}

Two simple lemmas below have been observed many times (e.g., \cite{rdim-zero-chen} or \cite[Sec.~3]{rdim-zero-elagin}). We include the proofs for completeness.

\begin{lemma}
  \label{lem: categories of rouquier dimension zero}
  Let $\mathcal{T}$ be an $\Ext$-finite idempotent-complete triangulated category. The following are equivalent:
  \begin{itemize}
  \item $\rdim \mathcal{T} = 0$;
  \item there are finitely many indecomposable objects in $\mathcal{T}$ up to isomorphisms and shifts.
  \end{itemize}
\end{lemma}
\begin{proof}
  If there are finitely many isomorphism classes of indecomposable objects up to shifts, let $M_1, \ldots, M_n$ be the list of representatives of those isomorphism classes. Then the direct sum $\oplus_{i=1}^n M_i$ is a generator of $\mathcal{T}$ with generating time zero.

  For the converse implication, let $G \in \mathcal{T}$ be a generator with generating time zero. There is a decomposition of $G$ into finitely many indecomposable direct summands $G_1, \ldots, G_n$. By assumption any object of $\mathcal{T} = \langle G \rangle_0$ is a direct summand of $G \otimes V_\bullet$ for some finite-dimensional graded vector space $V_\bullet$. By the Krull--Schmidt theorem any such direct summand is isomorphic to a direct sum of shifts of $G_i$'s. Therefore any indecomposable object of $\mathcal{T}$ is isomorphic to one of the $G_i$'s.
\end{proof}

If a triangulated category is generated by a single exceptional object, then it is equivalent to the derived category of vector spaces and clearly satisfies the assumptions of the lemma above. The key observation for Proposition~\ref{prop: apriori bounds for blow-ups} is that Lemma~\ref{lem: categories of rouquier dimension zero} applies to some more complicated categories as well.

\begin{lemma}
  \label{lem: dynkin quivers have rouquier dimension zero}
  Let $Q$ be a quiver whose underlying undirected graph is of ADE type. Then the bounded derived category of finite-dimensional representations $D^b(Q\mathrm{-rep})$ has Rouquier dimension zero.
\end{lemma}
\begin{proof}
  The path algebra of any quiver without relations has homological dimension one (see, for example, \cite[Prop.~1.4.1]{brion-quivers}). Therefore any object of $D^b(Q\mathrm{-rep})$ is quasiisomorphic to a direct sum of shifts of representations. In particular the only indecomposable objects are shifts of representations. Moreover, by Gabriel's theorem $Q$ has finitely many indecomposable representations up to isomorphism. So by Lemma~\ref{lem: categories of rouquier dimension zero} the Rouquier dimension of this category is zero.
\end{proof}

\section{Blowing up points and codimension two subspaces}
\label{sec: blowing up lemmas}

The existence of a semiorthogonal decomposition in some category leads to an upper bound on Rouquier dimension as in Lemma~\ref{lem: rouquier dimension under glueing}. The next proposition shows that sometimes this upper bound is preserved under blow-ups of points or codimension two subvarieties. When the upper bound is sharp, e.g., for $X \iso \P^2$, the blow-up also satisfies Orlov's conjecture.

\begin{proposition}
  \label{prop: apriori bounds for blow-ups}
  Let $X$ be a smooth $n$-dimensional variety. Assume that there is a semiorthogonal decomposition with $n-1$ exceptional line bundles $L_0, \ldots, L_{n-2}$:
  \[
    \Dbcoh(X) = \langle \mA, L_0, \ldots, L_{n-2} \rangle.
  \]
  Let $B$ be a set of cardinatly at most three, and let $\{ Z_b \}_{b \in B}$ be a set of subvarieties of $X$ such that any $Z_b$ is
  \begin{itemize}
  \item either a point; or
  \item a smooth codimension-2 subvariety such that the restrictions $\langle L_0|_{Z_b}, \ldots, L_{n-2}|_{Z_b} \rangle$ form a full exceptional collection in $\Dbcoh(Z_b)$.
  \end{itemize}
  Let $\pi\colon Y \to X$ be the blow-up morphism in the union of $\{ Z_b \}_{b \in B}$. Then there exists a semiorthogonal decomposition
  \[
    \Dbcoh(Y) = \langle \pi^*\mA, \mathcal{T}_0, \sldots, \mathcal{T}_{n-2} \rangle
  \]
  such that every subcategory $\mathcal{T}_i$ has Rouquier dimension zero.
\end{proposition}

\begin{remark}
  Lemma~\ref{lem: rouquier dimension under glueing} applied to the constructed semiorthogonal decomposition gives us a bound $\rdim(Y) \leq \rdim(\pi^*\mA) + n-1$. Since $\pi^*$ is fully faithful, this is exactly the same bound that we get for the Rouquier dimension of $X$ from the semiorthogonal decomposition of $\Dbcoh(X)$.
\end{remark}

We start by proving two lemmas, one about blow-ups of points and the other about blow-ups of codimension two subvarieties. Abusing the notation a little, we use the same letter to refer to a line bundle on the base and to its pullback to the blown up variety.

\begin{lemma}
  \label{lem: blowing up points}
  Let $X$ be a smooth $n$-dimensional variety. Assume that there is a semiorthogonal decomposition which includes $n-1$ exceptional line bundles $L_0, \ldots, L_{n-2}$:
  \[
    \Dbcoh(X) = \langle \mA, L_0, \ldots, L_{n-2} \rangle.
  \]
  Let $\pi\colon Y \to X$ be the blow-up of a point $x \in X$. Then there is a semiorthogonal decomposition
  \[
    \Dbcoh(Y) = \langle \pi^*\mA, \,\, L_0, \sldots, L_{n-2}, \, S_{n-2}, \sldots, S_0 \rangle
  \]
  such that for any $i$ the object $S_i$ is exceptional, it is supported set-theoretically on the exceptional divisor, and it satisfies $\RHom_Y(L_{i}, S_i) = k[0]$.
\end{lemma}

\begin{proof}
  Consider the dual to Orlov's decomposition of a blow-up (Proposition~\ref{prop: dual to orlov collection}):
  \[
    \Dbcoh(Y) = \langle \pi^*\mA, \,\, L_0, \sldots, L_{n-2}, \,\, \tau_{\geq -(n-2)}(\pi^*\O_x), \sldots, \tau_{\geq 0}(\pi^*\O_x) \rangle,
  \]
  where $\tau$ denotes the canonical truncation of a complex.
  We only need to check the identity $\RHom_Y(L_i, \tau_{\geq -i}(\pi^*\O_x)) = k[0]$ to finish the proof of the lemma. From Lemma~\ref{lem: cohomology sheaves for blow-ups} we know that the cohomology sheaves of $\pi^*\O_x$ are isomorphic to the pushforwards $j_*(\Omega^m(m))$ from the exceptional divisor $j\colon E \hookrightarrow Y$. If $L$ is any line bundle on $X$, then its pullback to $Y$ restricts trivially to $E$. By adjunction we have
  \[
    \RHom_Y(\pi^*L, j_*\Omega^m(m)) \caniso \RGamma(\P^{n-1}, \Omega^m(m)) =
    \begin{cases}
      0, & m \neq 0 \\
      k[0], & m = 0.
    \end{cases}
  \]
  Thus in particular $\RHom_Y(L_i, \tau_{\geq i} \pi^*\O_x) \caniso k[0]$, as claimed in the lemma.
\end{proof}

\begin{lemma}
  \label{lem: blowing up codimension two}
  Let $X$ be a smooth $n$-dimensional variety. Assume that there is a semiorthogonal decomposition which includes $n-1$ exceptional objects $L_0, \ldots, L_{n-2}$:
  \[
    \Dbcoh(X) = \langle \mA, L_0, \ldots, L_{n-2} \rangle.
  \]
  Let $Z \subset X$ be a smooth codimension-2 subvariety such that the restrictions $\langle L_0|_Z, \ldots, L_{n-2}|_Z \rangle$ form a full exceptional collection in $\Dbcoh(Z)$. Then there is a semiorthogonal decomposition for the blow-up $\pi\colon Y \to X$ of $X$ along the subvariety $Z \subset X$:
  \[
    \Dbcoh(Y) = \langle \pi^*\mA, \,\, L_0, \, \ldots \, , L_{n-2}, \, S_{n-2}, \, \ldots \, , S_0 \rangle
  \]
  such that for any $i$ the object $S_i$ is exceptional, it is supported set-theoretically on the exceptional divisor, and it satisfies $\RHom_Y(L_{i}, S_i) = k[0]$.
\end{lemma}

\begin{proof}
  Let $j\colon E \to Y$ be the embedding of the exceptional divisor, and let $p\colon E \to Z$ be the projection map. By Orlov's theorem \cite{orlov93} the functor $j_*p^*(-)\colon \Dbcoh(Z) \to \Dbcoh(Y)$ is a fully faithful embedding and there is a semiorthogonal decomposition:
  \[
    \Dbcoh(Y) = \langle \, \pi^*\mA, L_0, \ldots, L_{n-2}, \, j_*p^*(\Dbcoh(Z)) \, \rangle.
  \]
  
  By assumptions the category $\Dbcoh(Z)$ has a full exceptional collection $\langle L_0|_Z, \ldots, L_{n-2}|_Z \rangle$. Let $\langle \widetilde{L_{n-2}}, \ldots, \widetilde{L_{0}} \rangle$ be the (right) dual full exceptional collection in $\Dbcoh(Z)$. Consider the semiorthogonal decomposition
  \[
    \Dbcoh(Y) = \langle \, \pi^*\mA, L_0, \ldots, L_{n-2}, \,\, j_*p^*\widetilde{L_{n-2}}, \ldots, j_*p^*\widetilde{L_0} \, \rangle.
  \]
  Every object $j_*p^*\widetilde{L_i}$ is exceptional since the functor $j_*p^*$ is fully faithful. The restriction of a line bundle $L_i$ to $E \subset Y$ is isomorphic to $p^*(L_i|_Z)$, so by adjunction we have
  \[
    \RHom_Y(L_i, j_*p^*\widetilde{L_i}) \caniso \RHom_Z(L_i|_Z, \widetilde{L_i}) \caniso k[0]
  \]
  where the last isomorphism is a standard property of the dual exceptional collection (see, for example, \cite[Prop.~2.15]{kapranov88}). Thus all the conditions of the lemma are satisfied.
\end{proof}

Now we are ready to prove Proposition~\ref{prop: apriori bounds for blow-ups}.

\begin{proof}[Proof of Proposition~\ref{prop: apriori bounds for blow-ups}]
  For every subvariety $Z_b \subset X$ in the center of the blow-up choose exceptional objects $(S_b)_{n-2}, \ldots, (S_b)_0 \in \Dbcoh(Y)$ using Lemmas~\ref{lem: blowing up points} and \ref{lem: blowing up codimension two} respectively. The exceptional objects corresponding to distinct subvarieties are completely orthogonal since they are supported on disjoint exceptional divisors. Thus there is a semiorthogonal decomposition
  \[
    \Dbcoh(Y) = \langle \, \pi^*\mA, \, L_0, \sldots, L_{n-2}, \,\, \{ (S_b)_{n-2} \}_{b \in B}, \sldots, \{ (S_b)_0 \}_{b \in B} \, \rangle.
  \]
  Consider the subcategory $\mathcal{T} = \langle L_{n-2}, \{ (S_b)_{n-2} \}_{b \in B} \rangle \subset \Dbcoh(Y)$. It is generated by $1 + |B|$ exceptional objects. From Lemmas~\ref{lem: blowing up points} and \ref{lem: blowing up codimension two} we know that the only nontrivial morphism spaces in $\mathcal{T}$ are
  \[
    \RHom_Y(L_{n-2}, (S_b)_{n-2}) \iso k[0]
  \]
  for any $b \in B$. According to \cite[Th.~6.2]{bondal-exceptional} the category $\mathcal{T}$ is equivalent to the bounded derived category of finite-dimensional representations of the quiver
  
  
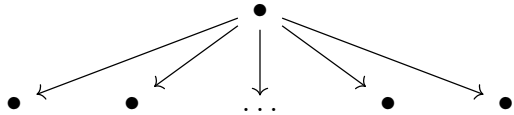
\[
    \begin{tikzcd}
      & & \bullet \arrow[ld] \arrow[lld] \arrow[d] \arrow[rd] \arrow[rrd] & & \\
      \bullet & \bullet & \ldots & \bullet & \bullet
    \end{tikzcd}
  \]
  with one source and $|B|$ target vertices. It is assumed that $|B| \leq 3$, so this is a quiver of Dynkin type: either $A_1$, $A_2$, $D_3$, or $D_4$. By Lemma~\ref{lem: dynkin quivers have rouquier dimension zero} in all those cases we have $\rdim \mathcal{T} = 0$, as claimed.

  Let $\mathcal{T}_0$ be the mutation of $\mathcal{T}$ through the subcategory $\langle \{ (S_b)_{n-3} \}_{b \in b}, \, \ldots\, , \{ (S_b)_{0} \}_{b \in B} \rangle$, i.e., there is a mutation of semiorthogonal decompositions of $\Dbcoh(Y)$:
  \[
    \begin{aligned}
      \langle \mA, \, L_0, \, \ldots \, , L_{n-3}, \, & \highlight{\mathcal{T}}, \, \{ (S_b)_{n-3} \}_{b \in b}, \, \ldots\, , \{ (S_b)_{0} \}_{b \in B} \rangle
      \rightsquigarrow \\
      & \rightsquigarrow
      \langle \mA, \, L_0, \, \ldots \, , L_{n-3}, \{ (S_b)_{n-3} \}_{b \in b}, \, \ldots\, , \{ (S_b)_{0} \}_{b \in B}, \highlight{\mathcal{T}_0} \rangle.
    \end{aligned}
  \]
  The category $\mathcal{T}_0$ is equivalent to $\mathcal{T}$, in particular it has Rouquier dimension zero. Now we can consider the subcategory $\langle L_{n-3}, \{ (S_b)_{n-3} \}_{b \in B} \rangle$, which will also be equivalent to the category of representations of the same quiver, and repeat the argument. Repeating this $n-2$ times finishes the proof of the proposition.
\end{proof}

\section{Blow-ups of projective spaces}
\label{sec: blow-ups of projective spaces}
Proposition~\ref{prop: apriori bounds for blow-ups} applies to varieties with many exceptional line bundles. Here we collect the implications for projective spaces.

\begin{theorem}
  \label{thm: blowing up projective spaces}
  Let $\{ Z_b \}_{b \in B}$ be a set of at most three disjoint linear subspaces of $\P^n$  such that every subspace $Z_b$ is either a point or has codimension two. Denote by $Y$ the blow-up of the projective space in the union $\sqcup_{b \in B} Z_b$. Then $\rdim Y = n$, i.e., $Y$ satisfies Orlov's conjecture.
\end{theorem}
\begin{proof}
  The projective space has a standard full exceptional collection \cite{beil78}:
  \[
    \Dbcoh(\P^n) = \langle \O, \O(1), \ldots, \O(n) \rangle.
  \]
  The restriction of $n-1$ last exceptional line bundles to any codimension-2 linear subspace is a full exceptional collection, so by Proposition~\ref{prop: apriori bounds for blow-ups} there is a semiorthogonal decomposition
  \[
    \Dbcoh(Y) = \langle \O, \O(1), \mathcal{T}_0, \ldots, \mathcal{T}_{n-2} \rangle
  \]
  where each component is a subcategory of Rouquier dimension zero. From Lemma~\ref{lem: rouquier dimension under glueing} we get the inequality $\rdim Y \leq n$. Since $Y$ is a smooth $n$-dimensional variety, $\rdim Y = n$.
\end{proof}

When $n$ is small, i.e., for $n=2$ and $n =3$, there are enough exceptional line bundles to repeat the procedure more than once.

\begin{proposition}
  \label{prop: surfaces satisfying orlov conjecture}
  Consider a tower of maps
  \[
    X_3 \to X_2 \to X_1 \to X_0 = \P^2
  \]
  where each map $\pi_i\colon X_i \to X_{i-1}$ is a blow-up in at most three distinct points. Then $X_3$ satisfies Orlov's conjecture, i.e., $\rdim X_3 = 2$.
\end{proposition}
\begin{proof}
  The derived category of $\P^2$ has a standard full exceptional collection: 
  \[
    \Dbcoh(\P^2) = \langle \O_{\P^2}, \O_{\P^2}(1), \O_{\P^2}(2) \rangle.
  \]
  By Proposition~\ref{prop: apriori bounds for blow-ups} after the first blow-up map $X_1 \to \P^2$ there is a semiorthogonal decomposition
  \[
    \Dbcoh(X_1) = \langle \pi_1^*\O_{\P^2}, \pi_1^*\O_{\P^2}(1), \mathcal{T}_1 \rangle
  \]
  such that $\rdim(\mathcal{T}_1) = 0$. A mutation of $\mathcal{T}_1$ through the two exceptional line bundles produces another decomposition:
  \[
    \Dbcoh(X_1) = \langle \widetilde{\mathcal{T}_1}, \pi_1^*\O_{\P^2}, \pi_1^*\O_{\P^2}(1) \rangle.
  \]
  Now we may apply Theorem~\ref{thm: blowing up projective spaces} to the second blow-up map $\pi_2\colon X_2 \to X_1$ choosing $\pi_1^*\O(1)$ as an exceptional line bundle. Then we repeat the same argument once again for the third blow-up morphism. The result is a semiorthogonal decomposition for $\Dbcoh(X_3)$ consisting of three subcategories of Rouquier dimension zero. By Lemma~\ref{lem: rouquier dimension under glueing} this establishes the upper bound $\rdim(X_3) \leq 2$. Since $X_3$ is a surface, this upper bound is sharp, $\rdim(X_3) = 2$.  
\end{proof}

\begin{corollary}
  Let $Y \to \P^2$ be a blow-up of up to nine arbitrary distinct points. Then $Y$ satisfies Orlov's conjecture, i.e., $\rdim Y = 2$.
\end{corollary}
\begin{proof}
  The variety $Y$ can be obtained via a tower of blow-ups as in Proposition~\ref{prop: surfaces satisfying orlov conjecture}.
\end{proof}

Note that this includes all del Pezzo surfaces. Orlov's conjecture for del Pezzo surfaces is due to Ballard and Favero \cite[Cor.~3.27]{ballard-favero}. They proved it by constructing tilting vector bundles with special properties. For del Pezzo surfaces our argument is shorter and does not require the points to be in general position, but Ballard and Favero also established the conjecture for some other surfaces to which our argument does not seem to apply, e.g., for Hirzebruch surfaces.

Interestingly, Ballard and Favero remark in \cite[Prop.~3.10]{ballard-favero} that after blowing up at least 11 distinct points on a plane tilting vector bundles with the desired properties do not exist, and the argument in Proposition~\ref{prop: surfaces satisfying orlov conjecture} does not apply to a blow-up of at least 10 points.

\begin{proposition}
  \label{prop: threefolds satisfying orlov conjecture}
  Consider a tower of maps
  \[
    X_2 \to X_1 \to X_0 = \P^3
  \]
  where each map $\pi_i\colon X_i \to X_{i-1}$ is a blow-up of a disjoint union of points and lines, at most three per level, where by a line we mean a strict transform of an one-dimensional linear subspace in $\P^3$. Then $X_2$ satisfies Orlov's conjecture, i.e., $\rdim(X_2) = 3$.
\end{proposition}
\begin{proof}
  For threefolds, applying Proposition~\ref{prop: apriori bounds for blow-ups} to a blow-up requires two exceptional line bundles, and the full exceptional collection on $\P^3$ has four. The assumption on the lines blown up by $\pi_2\colon X_2 \to X_1$ is sufficient to apply Lemma~\ref{lem: blowing up codimension two} for the second level of the blow-up tower. Thus we may use the same argument as in Proposition~\ref{prop: surfaces satisfying orlov conjecture}.
\end{proof}

It is possible that a similar technique may be applied to some other cases, such as a blow-up of $\P^4$ in a line, or even a disjoint union of three lines. However, since we were not able to use ADE quivers larger than $D_4$ in the proof of Proposition~\ref{prop: apriori bounds for blow-ups}, the method cannot be applied when the codimension is large enough, even potentially. For example, blowing up $\P^5 \subset \P^{14}$ will add 48 new exceptional objects to a category with 15 exceptional line bundles, and they cannot be grouped together into less than 16 copies of the $D_4$-quiver, which means that the bound from Lemma~\ref{lem: rouquier dimension under glueing} is not going to be sharp.

\appendix
\section{Dual to Orlov's collection}
\label{sec: dual to orlov collection}
Orlov \cite{orlov93} has constructed a semiorthogonal decomposition for a blow-up of a variety in a smooth center. In this appendix we calculate a certain mutation of that decomposition. For simplicity of notation we only consider a blow-up of a smooth point, but similar methods work in general. This result is well-known to experts, and we include the proof due to the lack of a convenient reference.

\begin{lemma}
  \label{lem: cohomology sheaves for blow-ups}
  Let $X$ be a smooth variety, $n := \dim X$. Consider the blow-up $\pi\colon Y \to X$ in a point $x \in X$, and let $j\colon \P^{n-1} \hookrightarrow Y$ be the inclusion of the exceptional divisor. For any $k \in \Z$ there is an isomorphism $\mathcal{H}^{-k}(\pi^*\O_x) \iso j_*(\Omega_{\P^{n-1}}^k(k))$.
\end{lemma}
\begin{proof}
  We may restrict to a neighborhood of $x \in X$ in which the point is the zero locus of a regular section of some vector bundle. Then the derived pullback $\pi^*\O_{x}$ may be computed using a Koszul resolution for the skyscraper sheaf $\O_{x}$. The pullback of that resolution to $Y$ is still a Koszul complex, and there exists a description of its cohomology sheaves in terms of the excess conormal bundle to $\P^{n-1} \subset Y$, see, e.g., \cite[Prop.~1.28]{sanna}. This description completes the proof.
\end{proof}

\begin{proposition}
  \label{prop: dual to orlov collection}
  Let $X$ be a smooth variety of dimension $n$. Let $\pi\colon Y \to X$ be the blow-up in a point $x \in X$. Denote by $S$ the derived pullback $S := \pi^*(\O_x)$ of the skyscraper sheaf at $x$. Then there exists a semiorthogonal decomposition:
  \[
    \begin{aligned}
      \Dbcoh(Y) \caniso \langle \, \pi^*\Dbcoh(X), \, \, \tau_{\geq -(n-2)}S, \, \ldots \, ,
      \,\, \tau_{\geq -1}S, \, \tau_{\geq 0}S \, \rangle.
    \end{aligned}
  \]
  Here for any $k \in [0, n-2]$ the canonical truncation $\tau_{\geq -k}S$ is an exceptional object.
\end{proposition}

\begin{proof}
  Let $j\colon E \to Y$ be the inclusion of the exceptional divisor $E \iso \P^{n-1}$. Denote the pushforward sheaf $j_*(\O_{\P^{n-1}}(k))$ by $\O_E$. There is a semiorthogonal decomposition for $Y$ constructed by Orlov \cite{orlov93}:
  \[
    \Dbcoh(Y) \caniso \langle \pi^*\Dbcoh(X), \O_E, \O_E(1), \ldots, \O_E(n-2) \rangle.
  \]
  Our decomposition is the dual to the part of Orlov's collection concentrated on the exceptional divisor. It can be obtained by a sequence of mutations described by the following claim.

  \begin{claim}
    For any $k \in [0, n-2]$ denote by $T_k$ the subcategory
    \[
      T_k := \langle \,\, \O_E, \O_E(1), \ldots, \O_E(k) \,\, \rangle \subset \Dbcoh(Y).
    \]
    Let $M$ be the mutation of $\O_E(k) \in T_k$ through the subcategory $T_{k-1} \subset T_k$:
    \[
      \begin{aligned}
        \langle \, \O_E, \O_E(1), \ldots, \O_E(k-1), \highlight{\O_E(k)} \, \rangle
        \rightsquigarrow
        \langle \, \highlight{M}, \O_E, \O_E(1), \ldots, \O_E(k-1) \, \rangle.
      \end{aligned}
    \]
    Then the exceptional object $M$ is, up to a shift, isomorphic to $\tau_{\geq -k}S$.
  \end{claim}

  A repeated application of this claim constructs the expected semiorthogonal decomposition. To prove the claim, it is enough to show three statements about the truncation $\tau_{\geq -k}S$:
  \begin{enumerate}
  \item[(a)] $\tau_{\geq -k}S \in T_k$;
  \item[(b)] $\tau_{\geq -k}S \in (T_{k-1})^\perp$;
  \item[(c)] $\tau_{\geq -k}S$ is not a direct sum of several copies of shifts of the same object.
  \end{enumerate}
  Indeed, any object in $T_k$ which lies in the right orthogonal to $T_{k-1}$ lies in the subcategory generated by the exceptional object $M$. This category is equivalent to the derived category of vector spaces, so if (c) also holds, then $\tau_{\geq -k}(\pi^*\O_x)$ is a shift of the generating exceptional object.

  The statement (c) easily follows from the description of cohomology sheaves of $S$ given in Lemma~\ref{lem: cohomology sheaves for blow-ups}. To prove the statement (a), note the inclusion of subcategories
  \[
    T_k = \langle \, \O_E, \O_E(1), \ldots, \O_E(k) \, \rangle \,\, \supset \,\, j_* \left(\,\, \langle \O_{\P^{n-1}}, \O_{\P^{n-1}}(1), \ldots, \O_{\P^{n-1}}(k) \rangle \subset \Dbcoh(\P^{n-1}) \,\, \right).
  \]
  On the projective space $E \iso \P^{n-1}$ it is easy to see that $\Omega^k(k) \in \langle \O, \O(1), \ldots, \O(k) \rangle$. Together with Lemma~\ref{lem: cohomology sheaves for blow-ups}, this implies that every cohomology sheaf of $\tau_{\geq -k}S$ lies in $T_k$. Since any object lies in the span of its cohomology sheaves, this proves (a).

  To deal with (b), we consider the truncation triangle
  \[
    \tau_{\leq -(k+1)}S \to S \to \tau_{\geq -k}S.
  \]
  Below we will show that both $\tau_{\leq -(k+1)}S$ and $S$ lie in $(T_{k-1})^\perp$, and therefore the same is true for the third object in the triangle.

  We start with the truncation $\tau_{\leq -(k+1)}S$. It is enough to show that every cohomology sheaf of that object lies in $(T_{k-1})^\perp$, and by Lemma~\ref{lem: cohomology sheaves for blow-ups} those sheaves are isomorphic to the pushforwards $j_*(\Omega^m(m))$ for all $m \in [k+1, n]$. The category $T_{k-1}$ is, by definition, generated by the pushforwards $j_*\O_E(l)$ for $l \in [0, k-1]$. For any choice of $l$ and $m$ we get
  \[
    \RHom_Y(j_*\O_E(l), j_*\Omega^m(m)) \caniso \RHom_{\P^{n-1}}(j^*j_*\O_{\P^{n-1}}(l), \Omega^m(m)).
  \]
  Since $j$ is an inclusion of a divisor, it is easy to compute that $j^*j_*\O(l) \iso \O(l) \oplus \O(l+1)[1]$. In our situation $m \geq l-2$, so on $\P^{n-1}$ both $\O(l)$ and $\O(l+1)$ are semiorthogonal to $\Omega^m(m)$. Therefore the space $\RHom_Y(j_*\O_E(l), j_*\Omega^m(m))$ vanishes for relevant values of $l$ and $m$, and hence $\tau_{\leq -(k+1)}S \in (T_{k-1})^\perp$.

  Consider now the object $S \iso \pi^*\O_x$. Since $\pi$ is a proper map between smooth varieties, it is well-known that there exists a right adjoint $\pi^!$ to the pushforward functor $\pi_*$, and it is given by the formula $\pi^!(-) = \pi^*(-) \otimes \omega_\pi$, where $\omega_\pi$ is the relative dualizing line bundle. In our case $\omega_\pi \iso K_Y \otimes \pi^*K_X^\dual \iso \O_Y((n-1)E)$. Thus
  \[
    \begin{aligned}
      \RHom_Y(j_*\O_E(l), \pi^*\O_x) & \caniso \RHom_Y(j_*\O_E(l) \otimes \omega_\pi, \, \pi^!(\O_x)) \caniso \\
      & \caniso \RHom_X(\pi_*(j_*\O_E(l) \otimes \O_Y((n-1)E)), \O_x)) \caniso \\
      & \caniso \RGamma(\P^{n-1}, \O_{\P^{n-1}}(l - n + 1))^\dual \otimes \RHom(\O_x, \O_x).
    \end{aligned}
  \]
  Since $l \leq k \leq n-2$, the line bundle $\O_{\P^{n-1}}(l-n+1)$ is acyclic, which means that $S \in (T_{k-1})^\perp$. This finishes the proof of the statement (b), and thus the proof of the entire proposition.
\end{proof}

\printbibliography

\end{document}